\documentclass{article} 
\usepackage{graphicx}
\usepackage{wrapfig}
\usepackage{caption}
\usepackage{multirow}
\usepackage{amsmath,amsfonts,bm}
\newcommand{\R}{\mathbb{R}}

\usepackage[font=small]{caption}

\usepackage{amsthm}
\usepackage{amssymb}
\usepackage{color}

\usepackage{ulem}
\usepackage{tikz}
\usetikzlibrary{arrows.meta, patterns}

\newtheorem{proposition}{Proposition}[section]

\newtheorem{definition}[proposition]{Definition}

\newcommand{\cD}{\mathcal{D}}

\newcommand{\SE}{\mathrm{SE}}

\newcommand{\regu}{\mathrm{reg}}
\newcommand{\beq}{\begin{equation}}
\newcommand{\eeq}{\end{equation}}
\newcommand{\lra}{\longrightarrow}

\newcommand{\dom}{D}

\newcommand{\xt}{{\tilde{x}}}
\newcommand{\yt}{{\tilde{y}}}
\newcommand{\xh}{x}
\newcommand{\yh}{y}

\newcommand{\bdot}{\dot b}
\newcommand{\vdot}{\dot v}
\newcommand{\fdot}{\dot f}
\newcommand{\thetadot}{\dot \theta}

\newcommand{\norm}[1]{\left\lVert#1\right\rVert}

\begin{document}

\centerline{\large \bf A new perspective on border completion in visual cortex}

\medskip
\centerline{\large \bf  
as bicycle rear wheel geodesics paths}

\medskip
\centerline{\large \bf via sub-Riemannian Hamiltonian formalism} 

\bigskip

\centerline{R. Fioresi\footnote{University of Bologna, Italy}\footnote{Funding by COST Action
CaLISTA CA21109, CaLIGOLA MSCA-2021-SE-01-101086123, GNSAGA}, 
A. Marraffa \footnote{University of Bologna, Italy},
J. Petkovic\footnote{Mainz University Medicine, Germany}}

\begin{abstract} 
We present a review of known models and a new simple mathematical modelling for 
border completion in the visual cortex V1 highlighting the
striking analogies with bicycle rear wheel motions in the plane.
\end{abstract}

\section{Introduction}\label{sec-intro}

The issue of perceptual completion has been widely studied in psychology from a phenomenological point of view. The Gestalt  psychology's Pr\"agnanz laws of perception  (see \cite{mather} and
refs. therein), identify the empirical principles under which the human brain integrates local and incomplete visual stimuli into shapes and smooth figures. Among them, continuity is the leading principle for the perception of contours and is reproduced as a law in any mathematical modeling of boundary completion.

A more quantitative approach to the study of visual phenomena and edge reconstruction, appears in \cite{hubel}, where the notion of receptive fields of simple and complex cells in the primary visual cortex is clearly introduced, and the functional structure of the primary visual cortex is also outlined \cite{wiesel}. Furthermore, the primary visual cortex, with its topographic, layered, and columnar functional organization, lends itself to fit into a differential geometry framework, through the concept of fiber bundles. This fact led to the first mathematical treatment of V1 in \cite{hoffmann} (see also \cite{bc} and refs. therein).
In particular, the experiments in \cite{field2} suggested the employment of geometrical methods to describe the border completion mechanism as in
\cite{mumford}. In \cite{petitot, citti} a natural sub-Riemannian structure on V1 is introduced through analytic methods, arriving to a rigorous notion of geodesic for the border completion problem. 
Later on more work in the same direction appeared in \cite{bc} and \cite{alexeevski}.

\medskip
The purpose of our paper is to provide a clear and organized overview of different treatments (e.g. \cite{citti, mumford, petitot, hoffmann}, \cite{alexeevski} and
refs. therein), highlighting, in particular, 
the striking analogy between human vision and bycicle kinematics (see \cite{mont}).
In fact, we tackle the problem of border completion, 
by exploiting the
analogies between the physiological
framework, that we established for border reconstruction, and the one described in \cite{mont} for bicycle rear wheel trajectories. In particular, taking advantage of this correspondence,
we are able to use the bicycle model to reformulate the problem of boundary completion in terms of finding solutions to the Hamilton's equations of the given sub-Riemannian structure in a more
geometric and intrinsic manner \cite{mont, montgomery}.

With this aim, in section \ref{sec-visual}, we give an insight into the neurophysiological aspects of the early visual system. We provide some elementary notions about the role of the retina and the lateral geniculate nucleus, and we outline the structural organization of the primary visual cortex. Here, we only focus on the aspects concerning oriented edge detection, and a more comprehensive description can be found in books on neural physiology (see \cite{Kandel} and refs. therein).

In section \ref{sec-compare}, we give a brief overview and comparison of models present
in the literature, relevant to our study.

In section \ref{sec-bike}, we recall the main features of the model for computing trajectories of the front and rear wheel of the bike through subriemannian Hamiltonian mechanics in \cite{mont} and \cite{montgomery}.

Finally, in section \ref{sec-geo} we highlight the analogies between the model for 
rear wheel bicycle trajectories and a model for the border detection mechanism in V1 through the \textit{ orientation vector field}. We are able to establish a dictionary that allows us to translate one problem into the other, and treat them as one from a mathematical point of view. 

Finally, in section \ref{sec-completion}, we exploit this identification to find the geodesic's equation of the subriemannian structure (\cite{citti, petitot}) in the Hamiltonian formalism (\cite{montgomery}). The solutions 
to the geodesic's equations provide us with the reconstructed boundaries in a more natural and
intrisic way.

\section{The visual pathway} \label{sec-visual}

In this section, we give a quick overview of the visual system's lowest portion, from the retina to the primary visual cortex; later on we shall outline a mathematical model for orientation construction which is faithful to the effective functioning of the neural structures involved. Hence, we need to address the question of how the perception of orientation arises from sensory receptors sensitive to single, discrete light inputs. With this aim, we focus our brief introduction on the retina, the lateral geniculate nucleus, and the primary visual cortex itself.

The first structure involved in the process of transforming images into nerve signals is the retina, which is located in the posterior inner surface of the eye. Roughly speaking, the eye captures the light input and projects it onto the retinal photoreceptors capable of measuring light intensity. At this early stage, the neural representation of the visual scene is still simple, with hyperpolarized receptors in lighted areas and depolarized receptors in dark areas.

Retinal ganglion cells (RGCs) project the information to the lateral geniculate nucleus in the thalamus through the optic nerve. Their main computational feature is their circular receptive field, divided into two concentric parts in mutual competition: the center and the inhibitory surround. As a consequence, RGCs give the best possible response to the difference in brightness between their center and surrounds, while they respond weakly to uniform light stimuli falling on the whole receptive field.

The lateral geniculate nucleus (LGN) is a small ventral projection of the thalamus connected to the optic nerve. It represents a crucial relay station in the image processing pathway, and its functionalities include a smoothing effect on the visual image, which is fundamental for contour perception and our modelling of the visual image as a smooth function. Its axons are directly sent to the primary visual cortex, where the construction of the orientation percept takes place. 

Along the lower section of the visual pathway, there is faithful preservation 
of spatial information. This property is called \textit{retinotopy} and, for the primary visual cortex V1, it means that we have a homeomorphism between the hemiretinal receptoral layer and V1 called \textit{retinotopic map}. 
The existence of such a map allows us to identify V1 with a  compact domain $D$ in $\mathbb{R}^2$.

In the primary visual cortex, we find cortical columns composed of simple and complex cells. These columns, through the non-concentric, striped receptive fields of their cells, build the orientation information from the input signal, effectively associating an orientation value, i.e. an angle, $\theta$ to each point $x,y$ of the perceived image.
\begin{figure}
\centering
\includegraphics[width=.9\textwidth]{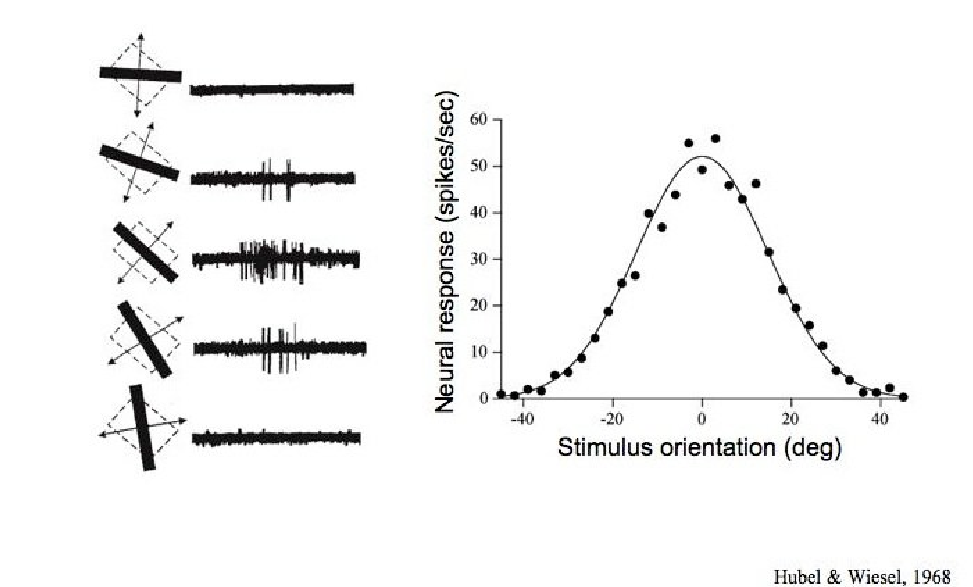}
\label{fig:}
\end{figure}
According to the fundamental work \cite{hubel}, V1 is organized in functional 
units called \textit{hypercolumns} that allow the analysis of different aspects of the visual information separately and in parallel. In particular, an orientation hypercolumn contains a full set of simple and complex cells, i.e. a full set of \textit{orientation columns}. Within an orientation column, we have cells which address the same portion of the visual image and best respond to the same orientation $\theta$.
\begin{figure}[!h]
    \centering
    \includegraphics[width=.75\textwidth]{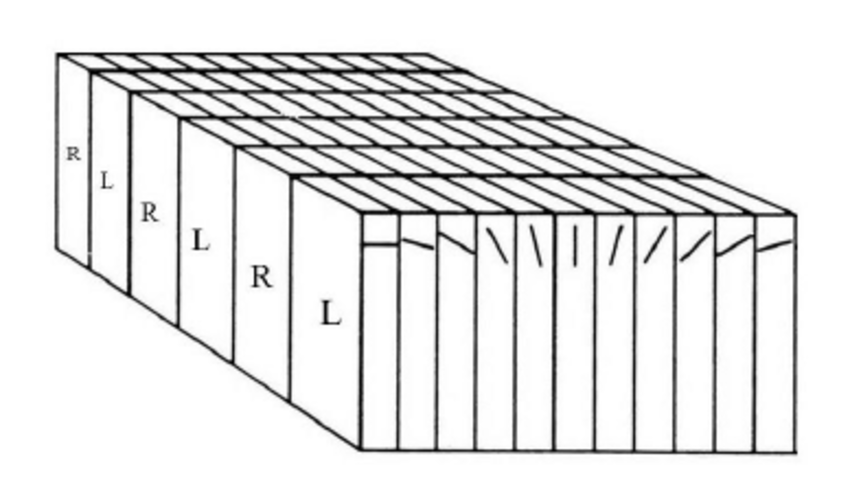}
    \captionsetup{width=.75\textwidth}
    \caption{Ice cube model}
    \label{fig:icecube}
\end{figure}

In this model, called the {\it ice cube model}, see
Fig. \ref{fig:icecube}, 
we assume that at each point of the domain $D$ there is a 
complete orientation hypercolumn, i.e at each point of the domain $D$ we can perceive any orientation from $0$ to $2\pi$. 
This fact leads to identifying V1 with the 
\textit{orientation bundle} defined as 
$\mathcal{E} := D \times S^1 \to D$, which also represents the configuration space manifold for border reconstruction (see \cite{citti}, \cite{petitot}, \cite{hoffmann}).

\section{Vision models for boundary completion}\label{sec-compare}
In this section we give a brief overview of the models for the boundary completion mechanism; for more details see \cite{citti2014} and refs. therein.

One of the first successful attempts to reproduce this phenomenon was reported by Mumford in \cite{mumford}. Here, starting from structural considerations on the V1 horizontal connectivity, the mechanism of border completion is modeled as a stochastic process where active hypercolumns induce a random compatible orientation in the neighbouring inactive hypercolumns, i.e. the hypercolumns corresponding to points in the image where a border is missing. In this framework, a boundary is completed when this inductive process generates a line joining two given orientation cues. Clearly, there is an infinite number of curves that can be obtained in this fashion, but Mumford assumes that the ones corresponding to the experimentally seen reconstructed boundaries are the ones that maximize the probability of reaching completion after a fixed number of iterations. This condition is then proven to be equivalent to minimizing the elastica integral 
 
\begin{equation*}
    \int_{\gamma} \alpha \, k^2+\beta \; ds 
\end{equation*}
where $\alpha$ and $\beta$ are constants,$k$ the curvature of $\gamma$, $ds$ the arc length. 
Such minimization leads to the \textit{elastica curve}, originally investigated by Bernoulli
and Euler \cite{euler} (for a review both mathematical and historical see \cite{Matsutani2010EulersEA} and refs. therin).
Later, Petitot and Tondut \cite{petitot} show how Mumford's approach can be reformulated in a differential geometric framework. Building upon the intuitions of Hoffman \cite{hoffmann}, who first provided an interpretation of the 
visual cortex as a contact bundle and described the visual psychololgical constancies as invariants of the action of the conformal group $CO(1,3)$ on such bundle, they formulate the boundary completion problem as a minimization of a suitable Lagrangian, and show that this approach is indeed equivalent to the elastica integral minimization.
Years later, Citti and Sarti (\cite{citti}) further expanded this geometrical approach, interpreting Hoffman's visual bundle as a sub-riemannian manifold, and the reconstructed boundaries as (horizontal) geodesics in this space. In addition, thanks to the fact that the distribution of such manifold is bracket-generating, they propose an iterative algorithm to compute the geodesics as solutions of the sub-heat equation
\begin{equation*}
    \partial_t u = \Delta_h u
\end{equation*}
where $\Delta_h$ is the sub-laplacian operator (a laplacian with derivations taken only along the horizontal directions). The existence and smoothness of $u$ is, indeed, granted by the bracked-generating distribution (\textit{H{\"o}rmander's condition}, see also \cite{citti2014} for a more comprehensive treatment and a complete bibliography).

In the next sections, we will show that our considerations, 
based on the analysis of the physiological structures of the visual
pathway,
can provide, with a geometric and intrinsec approach, a 
model for border completion, strictly related to the literature
ones as described above, and
leveraging a fascinating correspondence between the visual cortical structure and the paths traced by the rear wheel of a bike in motion.

\section{Bicycle Paths}\label{sec-bike}

In this section we briefly recall how to compute the 
trajectories of front and rear bicycle wheels with the
aid of subriemannian geometry in the Hamiltonian formalism.
Our main sources are \cite{mont} and \cite{montgomery}, where the
front wheel geodesics are explicitly computed.

\medskip
The configuration space of the front $f$ and rear wheel $b$ 
of a bicycle is given by:
\beq\label{wheels}
Q=\{(b,f) \in \R^2\times \R^2 \,| \, \|b-f\|= \ell\}
\eeq
where $\ell$ is
the distance between the rear and front wheels.
Assume 
$\ell=1$ as in \cite{mont}
Let us denote 
$$
v=f-b=(\cos \theta, \sin \theta) \quad \mapsto \quad (\xh,\yh)=
(\xt,\yt)-(\cos \theta, \sin \theta)
$$
where $b=(\xh,\yh)$, $f=(\xt,\yt)$. We are interested, differently
from \cite{mont}, in the motion of $b$, for reasons that will be
clarified later.
\begin{figure}
    \centering
    \includegraphics[width=.65\textwidth]{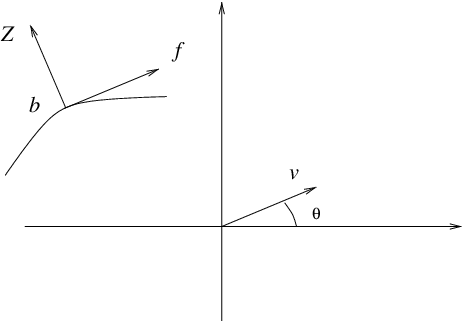}
    \caption{Configuration Space $Q$}
    \label{bike}
\end{figure}

We notice immediately that $Q\cong \R^2 \times S^1 \cong \SE(2)$. The
key importance of such identification, in the light of our previous sections
on the modeling of the visual cortex and border detection/reconstruction, will be fully elucidated in our
next section.

\medskip
The motion of $b$ is  subject
to the differential constraint 
$$
\bdot=kv,
$$
the so called {\sl no skid condition}, meaning that, when $b$ 
starts to move, it must be in the direction of $v$. We fix $k=1$,
an immaterial constant for our problem.
We observe that $\dot{v}$ is orthogonal to $v$, hence $\{v,\dot{v}\}$ is an orthonormal basis for $\mathbb{R}^2$.

Let $\fdot=\fdot_0+\fdot_\perp$ be
the decomposition of $f$ in such basis, hence:
$$
\fdot_\perp=\vdot=\fdot-\bdot=\fdot-\fdot_0=\fdot- \langle \fdot,v\rangle v
$$
the last equality is just expressing the projection onto $v$.
In coordinates, we obtain:
$$
(-\thetadot \sin \theta, \thetadot \cos \theta)=(\dot \xt, \dot \yt)-
\langle (\dot \xt, \dot \yt),(\cos \theta,  \sin \theta)  \rangle 
(\cos \theta, \sin \theta)
$$

A small calculation gives a differential constraint:
\beq\label{dc}
\thetadot - \cos\theta \dot \yt + \sin \theta \dot \xt=0
\eeq
If we substitute $f=b+v$ 
we get:
$$
\cos\theta \dot y - \sin\theta \dot x=0
$$
which corresponds to the contact form:
$$
\eta=\cos\theta  dy - \sin\theta dx 
$$
The kernel of $\eta$ gives the distribution spanned by:
$$
\cD = \mathrm{span} \{X=\partial_\theta, 
Y=\cos\theta \partial_x+\sin\theta\partial_y\}
$$
Its \textit{Reeb vector field} is 
\beq\label{reeb}
Z=\cos\theta \partial_y-
\sin\theta\partial_x
\eeq 
as one can readily check.

We now observe that the three vectors $X$, $Y$, $Z$ are left invariant
for the action of the special euclidean group $\SE(2)$, expressed
in the coordinates $(x,y,\theta)$.

\medskip
There is a natural metric on $\cD$ imposing the orthogonality of
the given generators: up to scalar multiplication this is the
only invariant subriemannian metric on $\SE(2)$ \cite{mont}.

\section{Border detection in V1 and bicycle paths}\label{sec-geo}

We want to establish a dictionary between concepts belonging
to very different scopes, visual system and bicycle paths,
that we introduced
in our previous sections to account for the striking
mathematical analogy, noticed already by the authors in \cite{mont}.
This will enable us to provide a simplified model for
the border completion, in our next section.

\medskip
In Sec. \ref{sec-visual}, we modelled the visual cortex (V1)
as the fiber bundle $\R^2 \times S^1$ over $\R^2$. Now we want to relate this
model with the configuration space $Q$ of bicycle
wheels described in
section (\ref{wheels}).

\medskip
We consider the Reeb vector field $Z$ (\ref{reeb}) on $\R^2 \times S^1$:
\begin{equation*}
    Z(x,y,\theta) \, =\, -sin\theta\, \partial_x + cos\theta\, \partial_y,
\end{equation*} 
Thanks to the action of the lateral geniculate nucleus, 
we can identify each image with a smooth function 
$I \colon D \to \mathbb{R}$, for a receptive field $D \subset \R^2$.
We can therefore define the \textit{orientation of a function}.

\begin{definition}
\label{def::theta}
{\rm Let $I \colon D \to \mathbb{R}$ be a smooth function, and $reg(D) \subseteq D$ 
the subset of the regular points of $I$.
We define the \textit{orientation map of $I$} as}
\begin{align*}
    \Theta \colon reg(D) &\to S^1 \\
    (x,y) &\mapsto \Theta(x,y) \, = \, 
\mathrm{argmax}_{\theta \in S^1} \{ Z(\theta)I(x,y)\}
\end{align*}
\end{definition}
We call $Z$ the \textit{orientation vector field}.

The map $\Theta$ is in charge of reproducing the behavior of 
simple and complex cells, so that an oriented edge is assigned at 
each point. We have that $\Theta$ is well defined. 

\begin{proposition}\label{border-prop}
Let $I:\dom \lra \R$ be as above and $(x_0,y_0)\in \dom$ a regular point for $I$.
Then, we have the following:
\begin{enumerate}
\item There there exists a unique $\theta_{x_0,y_0} \in S^1$ for which
the function $\zeta_{x,y}: S^1 \longrightarrow \R$, 
$\zeta_{x,y}(\theta):= Z(\theta) \, I(x,y)$ attains its maximum.
\item The map $\Theta:\regu(D) \longrightarrow  S^1$,
$\Theta(x,y)=\theta_{x,y}$
is well defined and differentiable.
\item The set:
$$
\Phi=\{(x,y,\Theta(x,y)) \in \dom \times S^1 : \Theta(x,y) = \theta_{x,y}\}
$$
is a regular submanifold of $\dom \times S^1$.
\end{enumerate}
\end{proposition}

\begin{proof}
(1). Since $\zeta_{x,y}$ is a differentiable function on a compact domain
it admits maximum, we need to show it is unique. We can explicitly
express:
$$
\zeta_{x,y}(\theta)=-\sin \theta \; \partial_xI+\cos \theta \; \partial_yI
$$
Since $(\partial_xI, \partial_yI)\neq (0,0)$ and it is constant,
by elementary considerations, taking the derivative of $\zeta_{x,y}$ with
respect to $\theta$ we see the maximum is unique.

(2). $\Theta$ is well defined by $(1)$ and differentiable.

(3). It is an immediate consequence of the implicit function theorem.

\end{proof}

Notice the following important facts:
\begin{itemize}
    \item the locality of the operator $Z(\theta)$  
mirrors the locality of the hypercolumnar anatomical connections;
    \item  its operating principle
is a good description of the combined action of simple and complex cells
(though different from their individual behaviour);
\end{itemize}

We can then view a smooth contour as the level set 
of some smooth function $I\colon D \to \mathbb{R}$; $\Theta$
is the orientation of the countour and the orientation vector field
$Z$, by construction, is orthogonal to such contour, since 
$\mathrm{argmax}_{\theta \in S^1} \{ Z(\theta)I(x,y)\}$ occurs when
$Z$ is aligned with $\nabla \, I$.

On the other hand,
as we can see from our Fig. \ref{bike}, $Z$, here the
Reeb vector, is orthogonal to $v$, which is tangent to the
trajectory of rear wheel $b$ by its very construction. 
Hence the angle we use to steer the handle of the bicycle, brings a
variation of the angle of the orientation vector field along
rear wheel path.  
So we have:

\medskip
\textit{Bicycle rear wheel paths described by the differential
constraint of the contact structure with the Reeb vector $Z$
coincide with the borders of an image detected in the visual cortex
with orientation vector $Z$}

\medskip
We report in the following table the dictionary between these
two different yet intriguingly related dynamics.

\begin{center}
  \begin{tabular}{  p{3cm} | p{3cm} | p{5cm} }
    \hline
Mathematics & Bicycle & Primary visual cortex\\ 
\hline
$b=(x,y)$ & position of the rear wheel & point in the image
of the retinotopic map
of the visual field \\
\hline
$v=f-b=(\cos\theta,\sin\theta)$ & direction of motion & detected 
orientation at $b$ 
\\ \hline
$(b,v)$ & point in configuration space & hypercolumn  \\ 
\hline
$Z=\dot{v}$
& normal to rear wheel path & orientation vector field
\\ \hline
$Q\cong \R^2 \times S^1 \cong \SE(2)$ & configuration space of
bicycle & V1 total space\\
    \hline
  \end{tabular}
\end{center}

\section{Rear wheel path and border completion}
\label{sec-completion}

Once the analogy is established in the dictionary of the previous
table, we can proceed and compute the geodesics of the
subriemannian structure defined in Sec. \ref{sec-bike} in the
Hamiltonian formalism. They give at once both 
the rear wheel trajectory and the border completion curve, given initial and final position and orientation of the reconstructed edge or, respectively, of the bicycle rear wheel.
We notice that, while in general Lagrangian and Hamiltonian
geodesics differ, here since we are in the special situation
of a fiber bundle $SE(2) \lra \R^2\cong SE(2)/S^1$, they coincide (see \cite{montgomery} Ch. 1).
Hence, with the Hamiltonian
formalism, we retrieve the same geodesic up to a (differential) phase factor 
as in \cite{citti}, hence strictly related to the treatment in 
\cite{mumford} (Sec. \ref{sec-compare})
in the visual cortex context. Besides we obtain the same equations as 
the bicycle paths in \cite{mont}, though the authors do not
fully compute such paths for the rear wheel, but for the front one
only (see also \cite{ms}).
Since we have the same differential constraint and the same contact form $\eta$, we have the same distribution \begin{align*}
    \cD\,= \mathrm{span} \{X=\partial_\theta, 
Y=\cos\theta \partial_x+\sin\theta\partial_y\}.
\end{align*} 
This distribution is bracket generating and, choosing the only invariant sub-Riemannian metric on $SE(2)$, we have $(Q,\cD,\langle,\rangle)$ sub-Riemannian manifold. 
Notice that both $X$ and $Y$ are left invariant vector fields for $SE(2)$,
so that our distribution is naturally invariant under the special euclidean group action.
Furthermore, since the uniqueness of the sub-Riemannian invariant metric, our construction
has this key natural invariance built-in and it appears more intrinsec than the equivalent descriptions
in Sec. \ref{sec-compare} (see \cite{mont,ms}).

\medskip
Following \cite{montgomery},  we can define a cometric for 
every $q \in \mathcal{E}$,
\begin{align*}
    \beta_q \colon T^*_q\mathcal{E} &\to T_q\mathcal{E} \\
    p &\mapsto \beta_q(p).
\end{align*} such that Im$\beta_q \, = \, \mathrm{span}\{X(q), Y(q)\}$.
Let us introduce the local coordinate system $(x,y, \theta, p_x, p_y, p_{\theta})$ on $T^*\mathcal{E}$, where $(p_x, p_y, p_{\theta})$ are the local coordinates on the cotangent bundle corresponding to $(x,y, \theta)$, defined by writing any covector $p$ as $p \, = \,p_x \, dx \, + \, p_y \, dy \, + \, p_{\theta} \, d\theta$.
We obtain the expression for the cometric in local coordinates
\[
\beta_q = \begin{pmatrix}
cos\theta & 0 \\
sin\theta & 0 \\
0 & 1
\end{pmatrix}
\]
The subriemannian Hamiltonian associated with $\beta$, is the functional
\begin{align*}
    H \colon T^*\mathcal{E} &\to \mathbb{R} \\
    (q,p) &\mapsto \frac{1}{2} \langle \beta_q(p), \beta_q(p) \rangle,
\end{align*} and  we know that we can express $H$ as
\begin{equation*}
    H \, = \, \frac{1}{2} ( P_X^2, P_Y^2).
\end{equation*} Here, $P_X$ and $P_Y$ are the momentum functions of $X$ and $Y$:
\begin{align*}
    P_X \, &= \, cos\theta \, p_x + sen\theta \, p_y\\
    P_Y \, &= \, p_{\theta},
\end{align*} where $p_x= P_{\partial_x}$, $p_y= P_{\partial_y}$, $p_{\theta}= P_{\partial_{\theta}}$, coincide with the local coordinates for $T^*_q\mathcal{E}$, defined above.
Hence, in local coordinates, we can write
\begin{equation*}
    H  \, = \, \frac{1}{2} [(cos\theta \, p_x \, + \, sin\theta \, p_y)^2 \, + p_{\theta}^2].
\end{equation*}\\
We know that to each Hamiltonian functional we can associate an Hamiltonian vector field, defined by the Hamilton equations. Furthermore, we know that $\dot{f}\,=\,\{f, H\}$, for any smooth function $f$ on the cotangent bundle.
If we define the auxiliary functions
\begin{align*}
    p_1 \, &= \, P_X\, =\, cos\theta\, p_x \,+\, sen\theta \, p_y \\
    p_2 \, &= \, P_Y \, = \, p_{\theta} \\
    p_3 \, &= \, P_Z \, = \, -sin\theta \, p_x \, + \, cos\theta \,p_y.
\end{align*}
We can obtain the Hamilton equations letting $f$ vary over the coordinate functions $(x, y, \theta)$, and the auxiliary functions on the cotangent bundle $(p_1, p_2, p_3)$.
Before going on with our derivation of the Hamilton equations, let us note that
\begin{equation*}
    \{P_X, P_Y\} \, = \,P_Z \, = \, -P_{[X,Y]}.
\end{equation*}
We can now easily calculate the Hamilton equations for the position coordinates $(x, y, \theta)$
\begin{align*}
    \dot{q}^i= \{q^i, H\} \longrightarrow \begin{cases}
    \dot{x}= cos\theta \, p_1 \\
    \dot{y} = sin\theta \, p_1 \\
    \dot{\theta} = p_2
    \end{cases}
\end{align*}
and for $(p_1 p_2, p_3)$
\begin{align*}
    \dot{p}_i= \{p_i, H\} \longrightarrow \begin{cases}
    \dot{p_1} = p_3 p_2 \\
    \dot{p_2} = - p_3 p_1 \\
    \dot{p_3} = - p_1 p_2
    \end{cases}
\end{align*}

As we know that the hamiltonian functional assumes a constant value along the hamiltonian flow (i.e, the solutions to Hamilton equations), we can write
\begin{equation*}
    E \, = \, p_1^2 + p_2^2,
\end{equation*} where $E/2$ is a constant hamiltonian value. Then, we can introduce an auxiliary variable $\gamma(t)$ such that
\begin{align*}
\begin{cases}
    p_1 \, &= \, \sqrt{E}\, sin(\frac{\gamma}{2}) \\
    p_2 \, &= \, \sqrt{E} \, cos(\frac{\gamma}{2}) \\
    p_3 \, &= \frac{1}{2} \, \dot{\gamma},
    \end{cases}
\end{align*} 
and therefore we can express the other variables as
\begin{equation}\label{eqgeo}
\begin{cases}
    \dot{x} \, &= \,\sqrt{E}\, sin(\frac{\gamma}{2}) \, cos\theta \\
    \dot{y} \, &= \, \sqrt{E}\, sin(\frac{\gamma}{2}) \, sin\theta \\
    \dot{\theta} \, &= \, \sqrt{E} \, cos(\frac{\gamma}{2}).
    \end{cases}
\end{equation}
The variable $\gamma$ satisfies a pendulum-like differential equation
\begin{equation*}
    \Ddot{\gamma} \, + \, E\,sin\gamma \, = \, 0,
\end{equation*} which is not analytically solvable. Hence, in order to give a graphical representation of the geodesic curves, we need to resort to numeric integration (see also \cite{ms}).
\begin{figure}
    \centering
    \includegraphics[width=.9\textwidth]{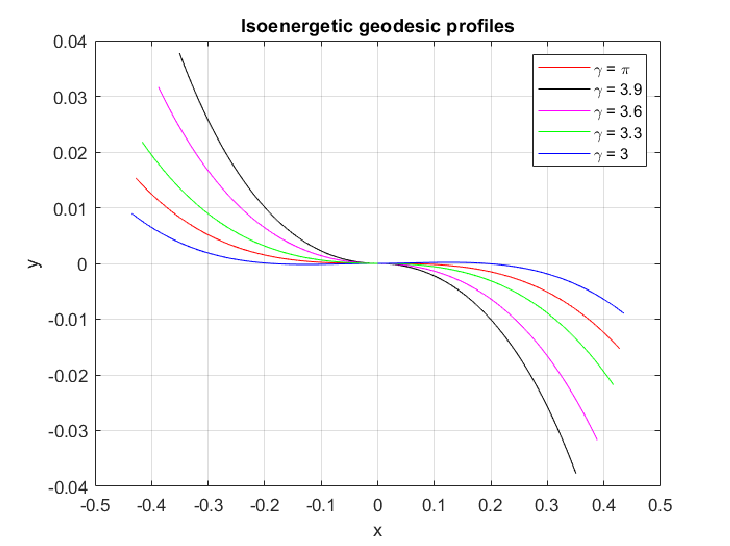}
    \caption{Solutions of the geodesic equations for some values of $\gamma$, projected onto the domain $D$. The energy is fixed to $E=0.2$.}
    \label{fig::isoenergetic}
\end{figure}
The solutions of these differential equations are the the lifts in $\mathcal{E}$ of perceived borders.
In Fig. \ref{fig::isoenergetic}, we show the projections on $D$ of some of the solutions with energy fixed to a value $E=0.2$, obtained varying the initial value of the parameter $\gamma$.
\begin{figure}[h!]
    \centering
        \hspace{-5.5mm}
        \includegraphics[width=0.5\textwidth]{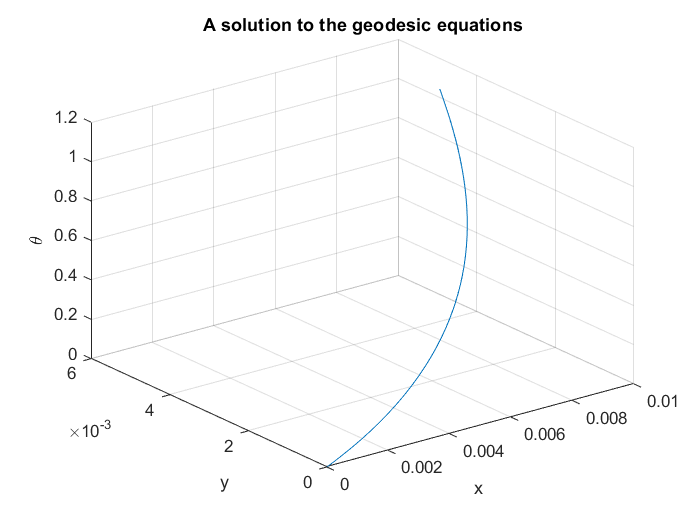}
        \hspace{-5.5mm}
    \caption{A solution to the geodesic equations that joins $(0,0,0)$ and $(0.01,0.005, \frac{\pi}{3})$ (on the right), and its projection onto the domain $D$ (on the left).}
    \label{fig::connect}
\end{figure}
A solution to the geodesic equations that joins the points $(0,0,0)$ and $(0.01,0.005, \frac{\pi}{3})$ is shown in Fig. \ref{fig::connect}, mimicking the border completion mechanism in the brain, given two boundary inducers, i.e, given initial and final coordinates $(x,y,\theta)$.

\begin{figure}[!h]
    \centering
    \includegraphics[width=\textwidth]{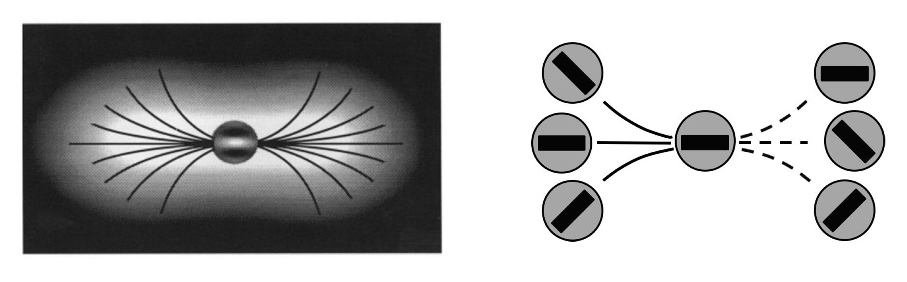}
    \captionsetup{width=\textwidth}
    \caption{The \textit{association field} (on the left): the rays extending from the ends of the central oriented element represent the optimal orientations at different positions. On the right, the specific rules of alignment are represented. \\- \cite{field}\\[10pt]}
    \label{fig::associationfield}
\end{figure}

Since we have a closed form for the vector field tangent to our solutions (see eq. (\ref{eqgeo})),
we can obtain an expression allowing us to make a direct comparison with the strictly
related treatments in \cite{citti, petitot} and also with the Euler's elastica as Mumford
introduces it for the border completion problem \cite{mumford} (see Sec. \ref{sec-compare}).
The curvature of the solution projection $\sigma = (x,y)$ is
obtained as: 
\begin{equation*}
    k_{\sigma}(t) = \frac{||\dot{\sigma}(t) \times \ddot{\sigma}(t)||}{\norm{\dot{\sigma}^3(t)}} = \frac{|\dot{\theta}|}{\sqrt{\dot{x}^2 + \dot{y}^2}} = \left| \cot{\frac{\gamma}{2}}\right|
\end{equation*}
We notice from (\ref{eqgeo}) that:
$$
\dot \Sigma=\dot x \partial_x + \dot y \partial_y + \dot \theta \partial_\theta=
\sin{\frac{\gamma}{2}}(X_1+k_\sigma X_2), \qquad 
X_1=\cos{\theta}\partial_x+\sin{\theta}\partial_y, \quad X_2=\partial_\theta
$$
We compare with the expression of $\dot \Sigma$ appearing in \cite{petitot}, which is equivalent
to the treatment in \cite{citti}:
\begin{equation}\label{pcs}
\dot{\Sigma} = X_1 + k X_2
\end{equation}
for the same vector fields $X_1$, $X_2$. We see the appearance of a \textit{phase factor}
compatible with the treatments in \cite{ms,Matsutani2010EulersEA}, due essentially to the
fact we look at the {\sl rear wheel}, while the solutions of the equation (\ref{pcs}) correspond
to the front one, leading to the elastica curve (this is a remark appearing also in \cite{mont}).

The Hamiltonian energy minimization along a curve $\Sigma$ can, therefore, be expressed in function of the projection curvature as follows:
\begin{equation*}
    \mathcal{E}(\Sigma) = \int_{\Sigma} E \, dt = \int_{\Sigma} E \sin^2{\frac{\gamma(t)}{2}} \, \left[1 + k_{\sigma}^2(t) \right]  \, dt
\end{equation*}
again, the same as in \cite{mumford}, up to the multiplicative factor $\sin^2{\frac{\gamma(t)}{2}}$
as we commented above.

We also mention that, though in general the question of sub-riemannian geodesics, may have different
answers in the Hamiltonian and Lagrangian formalism, since here we have a principal
bundle, due to the V1 modelling, there is a complete equivalence of the results obtained
in the two formalisms
(see \cite{montgomery}).

To conclude, we note the similarity of the geodesic curves depicted in fig. \ref{fig::isoenergetic} with the local \textit{association fields} from \cite{field} shown in fig. \ref{fig::associationfield}.
Fields, Heyes and Hess investigate, through a set of experiments, how the relative alignment of neighboring oriented elements is related to the perception of continuity in the human brain.
The information detected by single orientation selective cells is supposed to propagate locally through some \textit{long-range connections} in an orientation and position-specific modality (\cite{gilbert}, \cite{kapadia}).
According to our model, the natural interpretation of the local association field is the representation of the projection onto $D$ of a family of integral curves of the Hamiltonian vector field, i.e, the solutions to the geodesic equations corresponding to the joint constraints of position and orientation.

\section{Conclusions} 
Our comparison among the various V1 models
for border completion, available in the literature, shows
that with a simple mathematical modelling, taking advantage of the analogies
with the bicycle paths, 
allows us to obtain geodesics, which are strictly related to the ones
in \cite{petitot, citti, mumford}. 
Our treatment with hamiltonian formalism
is 
natural and establishes a dictionary, through low dimensional
subriemannian geometry, between visual path items and the bicycle
paths one as in \cite{mont}.

\section{Acknowledgements}
R. Fioresi wishes to thank Prof. Alekseevski, Prof. Citti, 
Prof. Sarti, Prof. Latini
and Prof. Breveglieri for helpful discussions and comments.

\bibliographystyle{plain}
\bibliography{bike-v1}



\end{document}